\documentclass[a4paper, 12pt, final]{amsart}

\makeatletter
\numberwithin{equation}{section}
\numberwithin{figure}{section}

\usepackage{amsmath,amsfonts,amssymb,amsthm,epsfig}

\usepackage{color}
\usepackage[final,linkcolor = blue,citecolor = blue,colorlinks=true]{hyperref}
\usepackage[leqno]{amsmath}

\usepackage{amscd}
\usepackage{mathrsfs}
\usepackage{}
\usepackage[all]{xy}
\usepackage[OT1]{fontenc}
\usepackage[latin1]{inputenc}
\usepackage{amssymb,latexsym}
\usepackage{type1cm,ae}
\usepackage[notref,notcite]{showkeys}
\usepackage{graphicx}
\usepackage{xspace}
\usepackage{setspace}
\usepackage[english]{babel}

\theoremstyle{definition}
		\newtheorem{theorem}{Theorem}[section]
				\newtheorem{proposition}[theorem]{Proposition}
				\newtheorem{lemma}[theorem]{Lemma}

                \newtheorem{example}[theorem]{Example}
                \newtheorem{question}[theorem]{Question}
                \newtheorem{conjecture}[theorem]{Conjecture}
				
     	        \newtheorem{definition}[theorem]{Definition}
	            \newtheorem{remark}[theorem]{Remark}
\numberwithin{equation}{section}

\newcommand*{\bR}{\ensuremath{\mathbb{R}}}

\newcommand*{\bC}{\ensuremath{\mathbb{C}}}

\newcommand*{\loc}{\mathrm{loc}}
\newcommand*{\closure}[1]{\overline{#1}}
\newcommand*{\bdary}[1]{\partial #1}

\newcommand*{\Wert}{\mathord{\mbox{|\kern-1.5pt|\kern-1.5pt|}}}
\newcommand*{\ie}{\mbox{i.e.}\xspace}

\newcommand{\R}{\mathbb{R}}	

\DeclareMathOperator{\dist}{dist}

\DeclareMathOperator{\diam}{diam}

\DeclareMathOperator{\modulus}{Mod}

\newcommand{\mbr}{{\mathbb{R}}}
\newcommand{\ra}{\rightarrow}
\newcommand{\mbd}{\mathbb D}
\newcommand{\mbs}{\mathbb S}
\newcommand{\mcg}{\mathcal G}
\newcommand{\mcf}{\mathcal F}

\def\XXint#1#2#3{{\setbox0=\hbox{$#1{#2#3}{\int}$}
  \vcenter{\hbox{$#2#3$}}\kern-.5\wd0}}

\makeatother

\usepackage{xcolor}
\definecolor{BLUE}{rgb}{0,0,1}

\usepackage{babel}
\begin{document}

\title{Generalized quasidisks and conformality: progress and challenges}
\dedicatory{Dedicated to Professor~Pekka Koskela on the occasion of his $60^\text{th}$ birthday}
\author{Chang-Yu Guo and Haiqing Xu}

\address[Chang-Yu Guo]{Research Center for Mathematics and Interdisciplinary Sciences, Shandong University 266237,  Qingdao, P. R. China}
\email{changyu.guo@email.sdu.edu.cn}

\address[Haiqing Xu]{Department of Mathematics and Statistics, University of Jyv\"askyl\"a, P.O. Box 35, FI-40014 University of Jyv\"askyl\"a, Finland}
\email{hqxu@mail.ustc.edu.cn}

\thanks{C.-Y.Guo is supported by the Qilu funding of Shandong University (No. 62550089963197). H. Xu is supported by the Academy of Finland (No. 21000046081).}

\begin{abstract}
In this notes, we survey the recent developments on theory of generalized quasidisks. Based on the more or less standard techniques used earlier, we also provide some minor improvements on the recorded results. A few nature questions were posed.
\end{abstract}
\maketitle

\tableofcontents


\section{Introduction}\label{sec:introduction}

\subsection{Quasidisks and conformality}
One calls a Jordan domain $\Omega\subset\bR^2$ a \textit{quasidisk} if it is
the image of the unit disk $\mathbb D$ under a quasiconformal mapping
$f\colon\bR^2\to\bR^2$ of the entire plane.
If $f$ is $K$-quasiconformal, we say that $\Omega$ is a $K$-quasidisk.
Another possibility is to require that $f$ is additionally conformal in
the unit disk $\mathbb D$. 

The following characterization, which shows that there is no real differences between these two definitions, is essentially due to K\"uhnau.
\begin{theorem}[\cite{k88,gh12}]\label{thm:Kuhnau}
A Jordan domain $\Omega\subset \R^2$ is a $K$-quasidisk if and only if $\Omega$ is the image of $\mathbb D$
under a $K^2$-quasiconformal mapping $f\colon \bR^2\to \bR^2$ that is conformal in $\mathbb D$.
\end{theorem}

The concept of a quasidisk is central in the theory of planar quasiconformal
mappings; see, for example, \cite{a06,aim09,gh12,lv73}. There are two well-known simple geometric characterization of quasidisks. The first one was given by Ahlfors.

\begin{theorem}[\cite{a63}]\label{thm:Ahlfors}
A Jordan domain $\Omega\subset \R^2$ is a quasidisk  if and only if it satisfies the three point property: 	
\begin{equation}\label{eq:Ahlfors three point property}
\min_{i=1,2}\diam(\gamma_i)\leq C|P_1-P_2|
\end{equation}
for any distinct pair of points $P_1,P_2\in \partial \Omega$, where $\gamma_1$ and $\gamma_2$ are the components of $\partial\Omega\backslash \{P_1,P_2\}$ and $C$ is a constant that depends on $\Omega$.
\end{theorem}

Another one is due to Gehring. 
\begin{theorem}[\cite{g82,gh12}]\label{thm:Gehring}
A Jordan domain $\Omega\subset \R^2$ is a quasidisk  if and only if it is linearly locally connected.
\end{theorem}
Recall that a domain $\Omega\subset \R^2$ is said to be linearly locally connected (LLC)
if there is a constant $C\geq 1$ so that
\begin{itemize}
\item (LLC-1) each pair of points in $B(x,r)\cap \Omega$ can be joined by an arc in
$B(x,Cr)\cap \Omega$, and
\item (LLC-2) each pair of points in $\Omega\backslash B(x,r)$ can be joined by an arc
in $\Omega\setminus B(x,C^{-1}r)$.
\end{itemize}

For more on the functions theoretic properties of a quasidisk; see the monograph \cite{gh12}.

\subsection{Generalized quasidisks}

A substantial part of the theory of quasiconformal mappings has recently been extended in a natural form to the setting of mappings of finite distortion with suitable integrability restrictions on the distortion function - particular with locally exponentially integrable distortion - see the monographs~\cite{im01,hk14} for a comprehensive overview. Here, we only briefly recall the basic definitions.

\begin{definition}\label{def:homeo of finite distortion}
We call a homeomorphism $f\colon\Omega\to f(\Omega)\subset\bR^2$ a homeomorphism of finite
distortion if $f\in W_{\loc}^{1,1}(\Omega,\bR^2)$ and
\begin{equation}\label{eq:disteq}
  \|Df(x)\|^2\leq K(x)J_f(x) \quad \text{almost everywhere in } \Omega,
\end{equation}
for some measurable function~$K(x)\geq1$ that is finite almost everywhere. 
\end{definition}
Recall here that
$J_f\in L_{\loc}^{1}(\Omega)$ for each homeomorphism $f\in W_{\loc}^{1,1}(\Omega,\bR^2)$ (see for instance~\cite{aim09}).
In the distortion inequality~\eqref{eq:disteq}, $Df(x)$ is the formal differential of~$f$ at the
point~$x$ and $J_f(x):=\det Df(x)$ is the Jacobian. The norm of~$Df(x)$ is defined as
\begin{equation*}
  \|Df(x)\|:=\max_{e\in\bdary{\mathbb D}} |Df(x)e|.
\end{equation*}
For a homeomorphism of finite distortion it is convenient to write $K_f$ for the optimal distortion
function. This is obtained by setting $K_f(x)=\|Df(x)\|^2/J_f(x)$ when $Df(x)$ exists and $J_f(x)>0$,
and $K_f(x)=1$ otherwise. The distortion of~$f$ is said to be locally $\lambda$-exponentially
integrable if $\exp(\lambda K_f(x))\in L_{\loc}^1(\Omega)$, for some $\lambda>0$.
Note that if we assume $K_f(x)$ to be bounded, $K_f\le K,$
we recover the class of $K$-quasiconformal mappings.
For this class, we have (see for instance \cite{aim09}) that
\begin{equation}\label{5.25-3}
f \in W^{1,p} _{\loc} (\Omega) \qquad \forall q <2K/(K-1).
\end{equation}

Following \cite{gkt13,g13}, we extend the definition of a quasidisk to the category of mappings of finite distortion, with an initial motivation to build a reasonable geometric counterpart for the theory of mappings with finite distortion. 
\begin{definition}[Generalized quasidisk]\label{def:generalized quasidisk}
A Jordan domain $\Omega\subset\bR^2$ is a generalized quasidisk if it is the image of the unit disk $\mathbb D$ under a homeomorphism $f\colon\bR^2\to\bR^2$ that is conformal in $\mathbb D$ and has locally exponentially integrable distortion in the entire plane.
\end{definition}

Another possibility for the definition of a generalized quasidisk is to remove the extra conformality requirement for the global homeomorphism $f$ in Definition \ref{def:generalized quasidisk} - we shall refer to the latter case a generalized quasidisk of second kind. However, unlike the case of a quasidisk, this leads to different classes of domains. Before turning to more details, we introduce two model domains that play an important role in understanding the geometry of a generalized quasidisk.

\begin{example}[Outward-pointing cusps]\label{exam:outward-pointing cusps}
For each $s>0$, the model outward-pointing cusp domain is given as
\begin{equation}\label{eq:cuspdomain}
  \Omega_s=\{(x_1,x_2)\in\bR^2:0<x_1<1,|x_2|<x_1^{1+s}\}\cup B(x_s,r_s),
\end{equation}
where $x_s=(s+2,0)$ and $r_s=\sqrt{(s+1)^2+1}$. 
\end{example}

\begin{example}[Inward-pointing cusps]\label{exam:inward-pointing cusps}
For each $s>0$, the model inward-pointing cusp domain is given as 
\begin{equation}\label{innercuap}
\Delta_s:= B(x'_s,r'_s)\setminus\{(x_1,x_2)\in\bR^2:x_1\ge 0,|x_2|\le
x_1^{1+s}\}
\end{equation}
where $x'_s=(-s,0)$ and $r'_s=\sqrt{(s+1)^2+1}.$
\end{example}

These cusp domains do not satisfy the Ahlfors three point property, and thus, they are not quasidisks. But they are generalized quasidisks according to Definition \ref{def:generalized quasidisk}. Indeed, we have the following result regarding the outward-pointing cusps.

\begin{theorem}[\cite{kt10,kt07}]\label{thm:outer cusp domain}
If $\lambda<\frac{2}{s}$, then there is a homeomorphism $f\colon \R^2\to \R^2$ with locally $\lambda$-exponentially integrable distortion such that $\Omega_s=f(\mathbb{D})$, while, this cannot happen if $\lambda>\frac{2}{s}$. Furthermore, if we require $\Omega_s$ to be a generalized quasidisk, then the above critical bound for $\lambda$ is $\frac{1}{s}$.
\end{theorem}

Notice the difference to the setting of quasiconformal mappings: instead of the switch from $K$ to $K^2$ under the additional
conformality condition, one essentially switches from $\lambda$ to $\lambda/2$. This type of conformality behavior disappears when we consider the inward-pointing cusps.

\begin{theorem}[\cite{gkt13}]\label{thm:inner cusp domain}
Given $\lambda<\frac{2}{s}$, there is a homeomorphism $f_s\colon\bR^2\to\bR^2$ of locally $\lambda$-exponentially
integrable distortion so that
$$f_s(\mathbb D)=
  \Delta_s:=  B(x'_s,r'_s)\setminus\{(x_1,x_2)\in\bR^2:x_1\ge 0,|x_2|\le
x_1^{1+s}\}  ,$$
where $x'_s=(-s,0)$ and $r'_s=\sqrt{(s+1)^2+1}.$ On the other hand, there is
no homeomorphism $g\colon\bR^2\to\bR^2$ of locally exponentially integrable distortion such
that $g$ is quasiconformal in $\mathbb D$ and $g(\mathbb D)= \Delta_s$.
\end{theorem}

In fact, if $g\colon\bR^2\to\bR^2$ is a homeomorphism of finite distortion $K_g(x)$ such that $g$ is
$K$-quasiconformal in $\mathbb D$ with $g(\mathbb D)=\Delta_s$, then it was proved in \cite{gkt13} that $K_g\notin L_{\loc}^p(\bR^2)$ if $p>K/s$.

Thus an inward pointing polynomial cusp rules out the extendability
of a Riemann mapping function to a homeomorphism of locally exponentially
integrable distortion, but such exterior cusps are  not that dangerous.

\subsection{Structure and notation}

We sometimes associate the plane $\bR^2$ with the complex plane $\bC$ for convenience and denote
by $\hat{\bC}$ the extended complex plane.
The closure of a set $U\subset\bR^2$ is denoted $\closure{U}$ and the boundary $\bdary{U}$.
The open disk of radius $r>0$ centered at $x\in\bR^2$ is denoted by $B(x,r)$ and we  simply write
$\mathbb D$ for the
unit disk.
The boundary of $B(x,r)$ will be denoted by $S(x,r)$ and the boundary of the unit disk $\mathbb D$
is written as $\partial \mathbb D.$
The symbol~$\Omega$ always refers to a domain, \ie a connected and open subset of~$\bR^2$.

When we write $f(x)\lesssim g(x)$, we mean that $f(x)\leq Cg(x)$ is 
satisfied for all $x$ with some fixed constant $C\geq 1$. Similarly, the 
expression $f(x)\gtrsim g(x)$ means that $f(x)\geq C^{-1}g(x)$ is satisfied 
for all $x$ with some fixed constant $C\geq 1$. We write $f(x)\approx g(x)$ 
whenever $f(x)\lesssim g(x)$ and $f(x)\gtrsim g(x)$.

This paper is structured as follows. 
In section \ref{sec:extension} we introduce a standard way to extend a conformal mapping to the whole plane.
Section \ref{sec generzlized qd} is about sufficient geometric conditions for generalized quasidisks.
Section \ref{con ext} is devoted to extension results of a particular class of quasiconformal mappings.

\section{Extension of a conformal welding}\label{sec:extension}
In this section, we briefly describe the standard way of extending a conformal map
$f\colon \mathbb D\to\Omega,$ where $\Omega$ is a Jordan domain, to a mapping
of the entire plane. 
First of all,
$f$ can be extended to a homeomorphism between $\closure{\mathbb D}$ and
$\closure{\Omega}$.
For simplicity, we denote this extended homeomorphism also by $f$.
It follows from the Riemann Mapping Theorem that there exists a
conformal mapping $g\colon \bR^2\setminus{\closure{\mathbb D}}\to\bR^2\setminus{\closure{\Omega}}$
such that the complement of the closed unit disk gets mapped to the complement
of $\overline \Omega$.
In this correspondence the boundary curve
$\Gamma=\bdary{\Omega}$ is mapped homeomorphically onto the boundary circle
$\partial \mathbb D$ and hence the composed mapping $G=g^{-1}\circ f$ is a
well-defined circle
homeomorphism, called \textit{conformal welding}.
Suppose we are able to extend $G$ to the exterior of the
unit disk, with the extension
still denoted by $G$. Then the mapping $G'=g\circ G$ will be well-defined
outside the unit disk and
it coincides with $f$ on the boundary circle $\partial \mathbb D$.
Finally, if we define
\begin{equation*}
  F(x)=
  \begin{cases}
    G'(x) & \text{if } |x| \geq 1 \\
    f(x) & \text{if } |x| \leq 1 ,
  \end{cases}
\end{equation*}
then we obtain an extension of $f$ to the entire plane.
In the case of a quasidisk, that is when  $\Omega$ is linearly locally
connected (LLC), the extension $G$ can be chosen to be quasiconformal
and hence the obtained map $F$ is also quasiconformal.

On the other hand, the extendability of a conformal mapping $f\colon \mathbb D\to \Omega$ to a homeomorphism $\hat f\colon \bR^2
\to \bR^2$ of locally integrable distortion is essentially equivalent to being able to extend the conformal welding $G'$ above to this class. Indeed, if $\hat f$ extends $f$, then $g^{-1}\circ \hat f$ extends $G$ to the exterior of $\mathbb D$ and has the same distortion as $\hat f.$ Reflecting (twice) with respect to the unit circle one then further obtains an extension to $\mathbb D\setminus \{0\}.$ Hence, one obtains an extension $\hat G'$ of $G'$ to $\bR^2\setminus \{0\}$ with distortion that has the same local  integrability degree as the distortion of $\hat f.$ If the latter distortion is sufficiently nice in a neighborhood of infinity (e.g. bounded), then this holds in all of $\bR^2$ as well.

Given a sense-preserving homeomorphism $f\colon \partial \mathbb D\to \partial \mathbb D$
and $0< t< \frac{\pi}{2}$, set
\begin{equation}
\delta_f(\theta,t)=\max \Big\{\frac{|f(e^{i(\theta+t)})-f(e^{i\theta})|}{|f(e^{i\theta})-f(e^{i(\theta-t)})|}, \frac{|f(e^{i(\theta-t)})-f(e^{i\theta})|}{|f(e^{i\theta})-f(e^{i(\theta+t)})|}\Big\}.
\end{equation}
Clearly $\delta_f$ is continuous in both variables, $\delta_f\geq 1$ and $\delta_f(\theta+2\pi,t)=\delta_f(\theta,t)$. The scalewise distortion of $f$ is defined as $\rho_f(t)=\sup_{\theta}\delta_f(\theta,t)$.

A well-known fact is that the extendability of a conformal welding $G\colon \bdary{\mathbb{D}}\to \bdary{\mathbb{D}}$ to a global homeomorphism of the entire plane with controlled distortion is related to the integrablity of $\rho_G$; see for instance~\cite[Section 4]{g13} for more information on this. For our purpose, we recall the following result, which is essentially due to Zakeri~\cite{z08}.
\begin{proposition}[\cite{z08,g13}]\label{coro:sufficient for extending}
Let $G\colon \bdary\mathbb{D}\to\bdary\mathbb{D}$ be a conformal welding. If 
\begin{equation*}
\rho_G(t)=O(\log\frac{1}{t}) \quad \text{as}\quad t\to 0,
\end{equation*}
then $G$ extends to a homeomorphism of the entire plane of locally exponentially integrable distortion. Furthermore, if 
\begin{equation*}
\rho_G(t)=O(t^{-\alpha}) \quad \text{as}\quad t\to 0
\end{equation*}
for some $\alpha>0$,
then $G$ extends to a homeomorphism of the entire plane of locally $p$-integrable distortion with any $p\in (0,\frac{1}{\alpha})$.
\end{proposition}

\section{Geometric criteria for generalized quasidisks}\label{sec generzlized qd}

In this section, we review known geometric criteria for a Jordan domain $\Omega$ to be a generalized quasidisk and present some improvement via basically the same techniques.

\subsection{Relaxing the three point property and linear local connectivity}
As observed from Theorems \ref{thm:outer cusp domain} and \ref{thm:inner cusp domain}, we have to relax the Ahlfors three point property or linear local connectivity  to include cusp domains. The extensions for these two concepts are straightforward.

\begin{definition}[Generalized three point property]\label{def:generalized three point property}
We say that a Jordan domain $\Omega\subset\bR^2$ satisfies the three point property with a control function $\psi$ if there exists a constant $C\geq 1$ and an increasing function $\psi\colon [0,\infty)\to [0,\infty)$ such that for each pair of distinct points $P_1,P_2\in\bdary\Omega$, 
\begin{equation}\label{eq:three point property with control function}
\min_{i=1,2}\diam(\gamma_i)\leq \psi\Big(C|P_1-P_2|\Big),
\end{equation}
where $\gamma_1,\gamma_2$ are the components of $\bdary\Omega\backslash\{P_1,P_2\}$. 
\end{definition}

\begin{definition}[Generalized local connectivity]\label{def:generalized local connectivity}
A domain $\Omega\subset \R^2$ is called $(\varphi,\psi)$-locally connected ($(\varphi,\psi)$-LC) if
\begin{itemize}
\item ($\varphi$-LC-1) each pair of points in $B(x,r)\cap\Omega$ can be joined by an arc in
$B(x,\varphi(r))\cap\Omega$, and
\item ($\psi$-LC-2) each pair of points in $\Omega\backslash B(x,r)$ can be joined by an arc in
$\Omega\setminus B(x,\psi(r))$,
\end{itemize}
where $\varphi, \psi:[0, \infty)\to[0, \infty)$ are smooth increasing functions such that
$\varphi(0)=\psi(0)=0$, $\varphi(r)\geq r$ and $\psi(r)\leq r$ for all $r>0$.
\end{definition}

For technical reasons, we assume that the function $t\mapsto \frac t {(\varphi^{-1}\circ \psi(t))^2}$ is decreasing
and that there exist constants $C_1,C_2$ so that $C_1\varphi(t)\le \varphi(2t)\le C_2\varphi(t)$
and $C_1\psi(t)\le \psi(2t)\le C_2\psi(t)$ for all $t>0.$
If $\varphi^{-1}=\psi$ above, $\Omega$ will simply be called $\psi$-LC. By \cite[Lemma 3.1]{g13}, a Jordan domain $\Omega\subset \R^2$ has the three point property with control function $\psi$ if and only if $\Omega$ is $\psi^{-1}$-locally connected.

Using the generalized three point property, the following result was proved in \cite{g13}.
\begin{theorem}[\cite{g13}]\label{thm:generalized quasidisk via genearlized three point property}
If a Jordan domain $\Omega\subset\bR^2$ has the three point property with the control function $\psi(t)=Ct\log^{s}\log(e+\frac{1}{t})$ for some positive constant $C$ and $s\in (0,\frac{1}{2})$, then $\Omega$ is a generalized quasidisk.
\end{theorem}

\begin{remark}\label{rmk:on PAMS2015 error}
In \cite[Theorem 1.1]{g13}, it was stated Theorem \ref{thm:generalized quasidisk via genearlized three point property} holds for $\psi(t)=\log^{\frac{1}{2}}(\frac{1}{t})$ and this is not correct according to the proofs given there. We shall present a more general result with corrected proofs.
\end{remark}

Via the nonlinear local connectivity, the following result was proved in \cite{gkt13}.

\begin{theorem}[\cite{gkt13}]\label{thm:generalized quasidisk via local connectivity}
Let $\Omega\subset\bR^2$ be a $(\varphi,\psi)$-locally
connected Jordan domain with
\begin{equation}\label{eq:sufficient condition}
    \lim_{r\to 0}\frac{r\cdot\varphi^{-1}\circ \psi(r)}{(\varphi^{-1}\circ\varphi^{-1}\circ\psi(r))^2\cdot \log\log\frac{1}{r}}=0,
\end{equation}
where $\varphi,\psi$ satisfy the technical conditions above. Then any conformal mapping $f\colon \mathbb D\to\Omega$  can be extended to the entire plane as a homeomorphism of locally exponentially integrable distortion.
\end{theorem}

Theorem \ref{thm:generalized quasidisk via local connectivity} implies in particular that if a Jordan domain $\Omega\subset\bR^2$ is $\psi$-locally connected with $\psi^{-1}(t)=Ct\log^s\log(e+\frac{1}{t})$ for some positive constant $C$ and some $s\in (0,\frac{1}{4})$, then $\Omega$ is a generalized quasidisk. Note that the range for $s$ is weaker than the one obtained in Theorem \ref{thm:generalized quasidisk via genearlized three point property}.

If $\Omega$ does not contain inward-pointing cusps, then we have the following weaker extension result.
\begin{theorem}[\cite{g13}]\label{thm:p extension case}
Let $\Omega\subset\bR^2$ be a LLC-1 Jordan domain. Then any conformal mapping $f\colon \mathbb D\to\Omega$  can be extended to the entire plane as a homeomorphism of locally $p$-integrable distortion for some $p>0$.
\end{theorem}

As the following examples indicates, similar result fails if we only assume $\Omega$ has no outward-pointing cusps.
\begin{example}\label{exam:inner cusp for sharpness}
Given any $\varepsilon>0$, there exists a Jordan domain $\Omega\subset \R^2$, which is LLC-2 and $\varphi$-LC-1 with $\varphi^{-1}(t)=Ct\big(\log\frac{1}{t}\big)^{-(1+\varepsilon)}$, while it fails to be a generalized quasidisk.
\end{example}

\begin{proof}
The example can be taken of the form
\begin{equation}\label{eq:cuspdomain2}
  \Omega= B((-1,0),\sqrt{5})\setminus\{(x_1,x_2)\in\bR^2:x_1>0,|x_2|<x_1\big(\log{e/x_1}\big)^{-(1+\varepsilon)}\}.
\end{equation}
It follows from~\cite[Theorem 6.2]{gkt13} that $\Delta$ is not a generalized quasidisk.
\end{proof}

Note that there is an extra logarithm gap between Theorem \ref{thm:generalized quasidisk via local connectivity} and Example \ref{exam:inner cusp for sharpness}.

We remark here that the generalized three point property is ``symmetric" in the sense that both inward-pointing and outward-pointing cusps are simultaneously allowed to have the same degree. While in~\cite[Theorem 5.1 and Theorem 6.1]{gkt13}, the degree of an inward pointcusp and an outward point cusp plays a different role in the extension result. This phenomenon is natural since polynomial interior cusps rule out the possibility of a locally exponentially integrable distortion extension, but polynomial exterior cusps do not.

On the other hand, from the technical view of point, doing moduslus of curve family estimates does not distinguish different types of cusps. Thus one expect a similar kind of ``symmetric" formulation appears as in~\cite{g13}. This is indeed the case; see Theorem~\ref{thm:general via nonlinear local connectivity} below.



\subsection{A more general result via nonlinear local connectivity}

\begin{theorem}\label{thm:general via nonlinear local connectivity}
Let $\Omega\subset\bR^2$ be a $(\varphi,\psi)$-locally connected Jordan domain with
\begin{align}\label{eq:sufficient condition}
\lim_{r\to 0}\frac{r}{\varphi^{-1}\circ \psi(r)\log\log\frac{1}{r}}= 0.
\end{align}
Then $\Omega$ is a generalized quasidisk.
\end{theorem}

A special case of Theorem \ref{thm:general via nonlinear local connectivity} can be formulated as follows: Let $\Omega\subset\bR^2$ be a Jordan domain, which is LLC-2 and $\varphi$-LC-1 with $\varphi^{-1}(t)=Ct\big(\log\log (1+\frac{1}{t})\big)^{-1+\varepsilon}$ for any $\varepsilon>0$. Then $\Omega$ is a generalized quasidisk. This is not sharp in view of Example \ref{exam:inner cusp for sharpness}.

For the proof of Theorem \ref{thm:general via nonlinear local connectivity}, we need the concept of the modulus of a curve family.
Recall that a Borel function $\rho\colon\bR^2\to[0,\infty\mathclose]$ is said to be admissible for a curve family $\Gamma$ if
$\int_\gamma\rho\,ds\geq1$ for each locally rectifiable $\gamma\in\Gamma$.
The modulus of the curve family $\Gamma$ is then
$$
  \modulus(\Gamma):=
\inf\Big\{\int_{\Omega}\rho^2(x)\,dx :  \rho
\text{ is admissible for }  \Gamma\}.
$$
For subsets $E$ and $F$ of $\closure{\Omega}$ we write $\Gamma(E,F,\Omega)$ for
the curve family consisting of all locally rectifiable paths joining~$E$ to~$F$
in~$\Omega$ and abbreviate $\modulus(\Gamma(E,F,\Omega))$ to $\modulus(E,F,
\Omega).$

We need two more lemmas for the proof of Theorem \ref{thm:general via nonlinear local connectivity}. The first one gives estimates of modulus of curve families inside a ball.
\begin{lemma}[\cite{v88}]\label{lemma:modululower}
Let $E, F$ be disjoint nondegenerate continua in $B(x,R).$ 
Then
\begin{equation}
C_0\frac{1}{\log(1+t)}\leq \modulus(E,F,B(x,R))\leq \tau(t)\footnote{The upper bound was incorrectly cited in \cite[Lemma 3.3]{g13} and thus it leads to the incorrectly stated Theorem 1.1},
\end{equation}
where $t=\frac{\dist(E,F)}{\min\{\diam E, \diam F\}}$, $\tau(t)\approx \log(1+\frac{1}{t})$ as $t\to 0$ and $C_0$ is an absolute constant.
\end{lemma}

The second lemma gives uniform continuity of quasiconformal mappings from locally connected domains onto the unit disk.

\begin{lemma}[\cite{kot01}]\label{lemma:modulus of continuity}
Suppose $g\colon \Omega\to \mathbb D$ is a $K$-quasiconformal mapping from a simply connected domain
$\Omega$ onto
the unit disk. Then there exists a positive constant $C$, (depending on $g$),
such that for any $\omega,\xi\in \Omega$,
\begin{equation}
|g(\omega)-g(\xi)|\leq Cd_{I}(\omega,\xi)^{\frac{1}{2K}},
\end{equation}
where $d_{I}(\omega,\xi)$ is defined as $\inf_{\gamma(\omega,\xi)\subset \Omega}\diam(\gamma(\omega,\xi))$.
In particular, if $\Omega$ above is $\varphi$-LC-1, then
\begin{equation}
|g(\omega)-g(\xi)|\leq C\varphi(|\omega-\xi|)^{\frac{1}{2K}}.
\end{equation}
\end{lemma}

\begin{proof}[Proof of Theorem \ref{thm:general via nonlinear local connectivity}]
The proof is a combination of~\cite[Proof of Theorem 5.1]{gkt13} and~\cite[Proof of Theorem 5.1]{g13}. Since $\Omega$ is a Jordan domain, $f$ extends to a homeomorphism between $\mathbb D$
and $\closure{\Omega}$ and we denote also this extension by $f$.
Let $e^{i(\theta-t)},\ e^{i\theta}$ and $e^{i(\theta+t)}$ be three points on $S$.
Since $f$ is a sense-preserving homeomorphism, $f(e^{i(\theta-t)}),\ f(e^{i\theta})$ and
$f(e^{i(\theta+t)})$ will be on the boundary of $\Omega$ in order.
Let $g\colon \bR^2\setminus {\closure {\mathbb D}}\to\bR^2\setminus{\closure{\Omega}}$
be a conformal mapping from the Riemann Mapping Theorem. Then $g$ extends to a homemorphism
between $\bR^2\setminus{\mathbb D}$ and $\bR^2\setminus{\Omega}$.
As before, we still denote this extension by $g$. Based on the discussion in the previous section, we only need to estimate the scale-wise distortion of the conformal welding $G:=g^{-1}\circ f$. 

Let $P=e^{i(\theta+\pi)}$ be the anti-polar point of $e^{i\theta}$ on $\bdary\mathbb{D}$ and let $\gamma_f(P,\theta-t)$ denote the arc from $f(P)$ to $f(e^{i(\theta-t)})$ on $\bdary\Omega$ . There exists a $t_0$ small enough such that $\diam(\gamma_f(P,\theta-t))\geq \diam(\gamma_f(\theta,\theta+t))$ when $t\in [0,t_0)$. Then by the proof of (5.2) in~\cite[Theorem 5.1]{gkt13}, there exists a constant $C_0>0$ such that
\begin{align*}
\diam(\gamma_f(\theta,\theta+t))\leq C_0\psi^{-1}\circ\varphi(d),
\end{align*}
where $\alpha_1=\gamma_f(\theta,\theta+t)$, $\alpha_2=\gamma_f(P,\theta-t)$ and $d=d(\alpha_1,\alpha_2)$.  Thus, it follows from Lemma~\ref{lemma:modululower} that
\begin{equation}\label{eq:modulus estimate for outer family}
\modulus(\Gamma')\leq C\log^{-1}\Big(1+\frac{\varphi^{-1}\circ\psi(\diam(\alpha_1))}{\diam(\alpha_1)} \Big),
\end{equation}
where $\Gamma'$ is the family of curves joining $\alpha_1$ and $\alpha_2$ in $\bR^2\backslash\closure{\Omega}$. Again by conformal invariance of modulus, we obtain that when $t$ is sufficiently small
\begin{equation}\label{eq:modulus from conformal invar}
\log(1+\delta_{G}(\theta,t))\leq C_0^{-1}\modulus(\Gamma'),
\end{equation}
where $C_0$ is the constant from Lemma~\ref{lemma:modululower}. 
Combining~\eqref{eq:modulus estimate for outer family} with~\eqref{eq:modulus from conformal invar} gives us the estimate
\begin{equation*}
\delta_{G}(\theta,t)\leq \exp\Big(\frac{C\diam(\alpha_1)}{\varphi^{-1}\circ\psi(\diam(\alpha_1))}\Big)+C_1.
\end{equation*}
On the other hand, by applying Lemma~\ref{lemma:modulus of continuity} and noticing that our technical assumptions on $\varphi$ implies that $\varphi^{-1}(t)\geq Ct^\alpha$ for some $\alpha>0$,
we obtain that
$$\diam(\alpha_1)\geq C\varphi^{-1}(t^2)\geq Ct^{2\alpha}.$$
Since $\frac{t}{\varphi^{-1}\circ \psi(t)}$ is non-increasing, we conclude that 
\begin{equation}\label{eq:important estimate for scalewise distortion}
\delta_{G}(\theta,t)\leq\exp\Big( \frac{Ct^{2\alpha}}{\psi^{-1}\circ\psi^{-1}(t^{2\alpha})}\Big)+C_1.
\end{equation}
Theorem~\ref{thm:general via nonlinear local connectivity} follows immediately from~\eqref{eq:sufficient condition},~\eqref{eq:important estimate for scalewise distortion} and Proposition~\ref{coro:sufficient for extending}.
\end{proof}

Comparing Theorem \ref{thm:general via nonlinear local connectivity} with Example \ref{exam:inner cusp for sharpness}, it is clear that there is an extra logarithm gap for the control function $\varphi$, even when the domain does not have inward-pointing cusps. We thus pose the following question for further research.

\begin{question}
Is Example \ref{exam:inner cusp for sharpness} sharp for domains without inward-pointing cusps? In other words, we would like to know what happens in between Theorem \ref{thm:general via nonlinear local connectivity} and Example \ref{exam:inner cusp for sharpness}.
\end{question}

\section{Further analytic aspect of the extension}\label{con ext}
In this section, we study further analytic aspect of the extension problem for quasiconformal mappings $g\colon \mbd \ra \Delta_s,$ where $\Delta_s$ is defined in \eqref{innercuap}. We start with an extension result for the case when $g$ is conformal.

\begin{theorem}[\cite{x19}]\label{theorem K_f}
Let $g$ be a conformal map from $\mathbb{D}$ onto $\Delta_s$. 
Then there is a homeomorphic extension $f :\mbr^2 \ra \mbr^2$ of $g$ with finite distortion. Moreover we have that
\begin{equation}\label{theorem |df|}
f \in W^{1,p} _{\loc} (\mbr^2 ,\mbr^2) \qquad \forall p<\infty,
\end{equation}
\begin{equation}\label{theorem f^-1 : 1} 
f^{-1} \in W^{1,p} _{\loc}(\mathbb{R}^2 , \mbr^2) \qquad \forall p< \frac{2(s+2)}{2s+1},
\end{equation}
\begin{equation}\label{theorem f^-1 : 2} 
 K_{f^{-1}} \in L^q _{\loc} (\mathbb{R}^2) \qquad \forall q<  \frac{s+2}{s} .
\end{equation}
\end{theorem}

Note that the above $\Delta_s $ is slightly different from that defined in \cite{x19}.  However, the proof of Theorem 1.2 from \cite{x19} can be modified in a straightforward manner to yield Theorem \ref{theorem K_f}. For the convenience of the readers, we briefly sketch the proof of $\mathcal{F}_s (g)\neq \emptyset $ in Theorem \ref{theorem K_f}.
After some simple reduction (composing with additional M\"obius transformations), it suffices to prove that $\mathcal{F}_s (g_0)\neq \emptyset $ for a fixed conformal mapping $g_0\colon \mbd \ra \Delta_s .$
Via the bi-Lipschitz characterization of chord-arc domains \cite{Tukia 1980}, it is easy to construct an extension of $g_0$ on any region that is strictly away from the cusp point. Thus the essential task is to construct an extension of $g$ in a small neighbourhood containing the cusp point. In this step, one can write down the extension by hand using the explicit geometry of $\Delta_s$ (see Step 1 in \cite[Proof of Theorem 1.2]{x19}). Combining these two extensions leads to an element in the set $\mathcal{F}_s(g_0)$. 

Compared with Theorem \ref{theorem K_f}, there are more delicate regularity results in \cite{x19}.
We next explore analogous results when $g\colon \mbd \ra \Delta_s $ is more generally a quasiconformal mapping.
To this end, for $s \in (0,\infty)$ and $K \in (1,\infty),$ we set 
\begin{equation}\label{g_sK}
\mcg_s (K) =\{g : g\colon \mbd \ra \Delta_s \mbox{ is a }K \mbox{-quasiconformal mapping from }\mbd \mbox{ onto } \Delta_s\}.
\end{equation}
Obviously $\mcg_s (K) \neq \emptyset,$ since there are  conformal mappings  in $\mcg_s (K) .$ 
Given any $g \in \mcg_s (K),$ we set
\begin{equation}\label{f_sg}
\mcf_s (g) = \{f : f\colon \mbr^2 \ra \mbr^2 \mbox{ is a homeomorphic extension of } g \mbox{ with finite distortion}  \}.
\end{equation}
For $a \in \mbr$ we denote $a_+ = (a+|a|)/2 .$ 
We have the following result that partially extends \cite[Theorem 1.2]{x19}.

\begin{theorem}\label{thm1}
Given any $g \in \mcg_s (K) ,$ we have that $\mcf_s (g) \neq \emptyset $ and 
\begin{equation}\label{supD_f}
\sup \{p \in [1,\infty) : f \in \mcf_s (g) \cap  W^{1,p} _{\loc} (\mbr^2 , \mbr^2)  \} 
 =  \frac{2K}{K-1}. 
\end{equation}
Moreover, we have
\begin{equation}\label{infsupD_f^-1}
\inf_{g \in \mcg_s (K)} \sup \{p \in [1,\infty) : f \in \mcf_s (g),\ f^{-1} \in  W^{1,p} _{\loc} (\mbr^2 , \mbr^2)  \} 
 =  \frac{2(2+s)}{2s +2 -K^{-1}} ,
\end{equation}
\begin{align}\label{supsupD_f^-1}
& \sup_{g \in \mcg_s (K)} \sup \{p \in [1,\infty) : f \in \mcf_s (g),\ f^{-1} \in  W^{1,p} _{\loc} (\mbr^2 , \mbr^2)  \} \notag \\
 =  & \min \left\{\frac{2K}{K-1},\frac{2(2+s)}{(2s +2 -K)_+} \right \} .
\end{align}
\end{theorem}

We need a couple of auxiliary results for the proof of Theorem \ref{thm1}.
In the first lemma, we provide a standard method to extend mappings in $\mcg_s (K).$  
\begin{lemma}\label{lem5.25}
For any $g \in \mcg_s (K),$ there exists an $f \in \mcf_s (g)$ such that
\begin{equation}\label{lem5.25-1}
f \in W^{1,p} _{\loc} (\mbr^2 ,\mbr^2) \mbox{ for all } p<\frac{2K}{K-1} ,
\end{equation}
and 
\begin{equation}\label{lem5.25-2} 
f^{-1} \in W^{1,q} _{\loc} (\mbr^2 ,\mbr^2) \mbox{ for all } q<\frac{2(2+s)}{2s +2 -K^{-1}}. 
\end{equation}
\end{lemma}

\begin{proof}
Fix a conformal mapping $\varphi\colon \mbd \ra \Delta_s $ satisfying $\varphi(\bar{z}) = \overline{\varphi(z)}$ for all $z \in \mbd.$
By Theorem \ref{theorem K_f} there is a homeomorphic extension $\Phi\colon \mbr^2 \ra \mbr^2$ of $\varphi$ with finite distortion.
Given any $g \in \mcg_s (K) ,$ we set $\psi  =  \varphi^{-1} \circ g  .$ Then $\psi\colon \mbd \ra \mbd$ is a $K$-quasiconformal mapping. Via reflection, we obtain a $K$-quasiconformal extension $\Psi\colon \mbr^2 \ra \mbr^2$
of $\psi .$   
Set 
\begin{equation*}
f = \Phi \circ \Psi .
\end{equation*}
To show $f \in \mcf_s(g),$ it suffices to check that $f \in W^{1,1} _{\loc} (\mbr^2 ,\mbr^2) .$
Alternatively, this will be done if we can prove \eqref{lem5.25-1}.

To this end, we let $p_1 \in (1,\infty)$ and $p_2 \in \left(0, \frac{2p_1 K} {(p_1 -1)(K-1)} \right) .$
Set $p$ by $p^{-1} = (2 p_1)^{-1} +p^{-1} _2 .$
Via the monotonicity we have that
\begin{equation}\label{5.25-1}
p  < \frac{1}{\frac{1}{2p_1} +\frac{(p_1 -1)(K-1)}{2 p_1 K}}
 < \frac{2K}{K-1}.
\end{equation}
From the chain rule and the Lusin ($N^{-1}$) property of $\Psi ,$ it follows that $Df(z)$ exists and
$$Df (z) = D \Phi (\Psi (z)) D \Psi(z)$$
for almost every $z \in \mbr^2 .$
Fix an arbitrary compact set $M \subset \mbr^2 .$
By H\"older's inequality, we have
\begin{equation}\label{infsupD_f3}
 \int_{M} |Df|^p 
\le \left(\int_{M} |D \Phi (\Psi)|^{2p_1} |D \Psi |^{2} \right)^{\frac{p}{2p_1}}
\left(\int_{M}  |D \Psi |^{\frac{(p_1-1)p_2}{p_1}} \right)^{\frac{p}{p_2}} .
\end{equation}
On the one hand, via the area formula and \eqref{theorem |df|} we obtain that
\begin{equation*}\label{infsupD_f4}
\int_{M} |D \Phi (\Psi)|^{2p_1} |D \Psi |^{2} 
\approx \int_{M} |D \Phi (\Psi)|^{2p_1} J_{\Psi} 
\le  \int_{\Psi(M)} |D \Phi |^{2p_1}<\infty.
\end{equation*}
On the other hand, note that $(p_1-1)p_2/p_1 <2K /(K-1).$
Hence via \eqref{5.25-3} we have that $\int_{M}  |D \Psi |^{(p_1-1)p_2 /p_1} <\infty .$
Therefore from \eqref{infsupD_f3} and \eqref{5.25-1} we obtain that $f \in W^{1,p} _{\loc} (\mbr^2 ,\mbr^2) $ for all $p< 2K/(K-1).$

It remains to show \eqref{lem5.25-2} and the proof is analogous to that of \eqref{lem5.25-1}.
Let 
$q_1 \in \left(1,\frac{K}{K-1} \right),$
$q_2 \in \left(0,\frac{p_1 (s+2)}{s+1} \right),$
$q_3 \in \left(0,  \frac{2 p_1 (s+2)}{(p_1 -1) (2s+1)}\right).$
Define $q$ by $q^{-1} = q^{-1} _2 +q^{-1} _3 .$
By monotonicity, we have that 
\begin{equation}\label{infsupD_f^-1 6-1}
q < \frac{1}{\frac{s+1}{q_1 (s+2)} + \frac{(q_1 -1) (2s+1)}{ 2 q_1 (s+2)}}
<\frac{2(s+2)}{2s+2-K^{-1}}.
\end{equation}
Fix a compact set $M \subset \mbr^2 .$
H\"older's inequality implies that
\begin{align}\label{infsupD_f^-1 6}
\int_{M} |Df^{-1}|^q 
\le & \left( \int_{M} |D \Psi^{-1} (\Phi^{-1} )|^{q_2}  |D \Phi^{-1} |^{\frac{q_2}{q_1}}  \right)^{\frac{q}{q_2}}  
\left(\int_{M} |D \Phi^{-1} |^{(1-\frac{1}{q_1})q_3} \right)^{\frac{q}{q_3}} \notag \\
=: & I^{\frac{q}{q_2}}  J^{\frac{q}{q_3}}. 
\end{align}
On the one hand, by \eqref{theorem f^-1 : 1} and the fact that $(1-\frac{1}{q_1})q_3 < 2(s+2)/(2s+1) $ we easily conclude that
$J <\infty. $
On the other hand, by H\"older's inequality, we have
\begin{equation}\label{infsupD_f^-1 7}
I \le  \left( \int_{M} |D \Psi^{-1} (\Phi^{-1} )|^{2 q_1} J _{\Phi^{-1}} \right)^{\frac{q_2}{2 q_1}}
\left(\int_{M} K^{\frac{q_2}{2 q_1 -q_2}} _{\Phi^{-1}} \right)^{1- \frac{q_2}{2 q_1}} .
\end{equation}
Note that $2q_1 < 2K/(K-1) .$
Hence via the area formula and \eqref{5.25-3} we infer that 
$\int_{M} |D \Psi^{-1} (\Phi^{-1} )|^{2 q_1} J _{\Phi^{-1}} 
< \infty.$
Since $q_2 /q_1 <(s+2)/(s+1) ,$ we obtain that $q_2/(2 q_1 -q_2)< (s+2)/s.$ 
Hence  \eqref{theorem f^-1 : 2} implies that  
$\int_{M} K^{q_2 /(2 q_1 -q_2)} _{\Phi^{-1}} <\infty.$
Therefore $I$ as in \eqref{infsupD_f^-1 7} is finite.
From \eqref{infsupD_f^-1 6-1} and \eqref{infsupD_f^-1 6}
we obtain \eqref{lem5.25-2}.
\end{proof}

In the next step, we construct technically two quasiconformal mappings from $\mbd$ onto $\Delta_s .$ 

\begin{example}\label{prop build qc}
For any $K \in (1,\infty),$ there is a $K$-quasiconformal mapping $g\colon \mbd \ra \Delta_s $ satisfying that $g(\bar{z}) = \overline{g(z)}$ for all $z \in \mbd $ and 
\begin{equation}\label{prop build qc1}
\diam (g (\mbs^1 \cap B(e^{i \pi},r))) \approx r^{2K}
\end{equation}
whenever $r \ll 1.$
\end{example}

\begin{proof}
Let $\varphi_1 (z) = |z+1|^{K-1} (z+1)$ for $z \in \mbd.$
Obviously $\varphi_1$ is a $K$-quasiconformal mapping.
Denote $D_K = \varphi_1 (\mbd).$ Then $D_K$ is symmetric with respect to the real axis.
Moreover $\partial D_K$ is piecewise smooth without cusp points.
Hence $D_K$ is a chord-arc domain. As a consequence, via \cite{Tukia 1980}, there is a bi-Lipschitz mapping 
$\varphi_2\colon D_K \ra \mbd$
satisfying that $\varphi_2 (\bar{z}) = \overline{\varphi_2 (z)}$ for all $z \in D_K $ and $\varphi_2 (0) =e^{i \pi}.$ 
By the arguments in \cite[Subsection 2.3]{x19}, there exists a conformal mapping 
$\varphi_3\colon \mbd \ra \Delta_s$
such that
$\varphi_3 (\bar{\xi}) = \overline{\varphi_3(\xi)}$ for all $\xi \in \mbd,$ $\varphi_3 (e^{i \pi})=0$ and 
\begin{equation}\label{prop build qc3}
\diam (\varphi_3 (\mbs^1 \cap B(e^{i \pi},t))) \approx t^{2}
\end{equation}
whenever $t \ll 1.$ 

Set $g = \varphi_3 \circ \varphi_2 \circ  \varphi_1 .$ Then $g\colon \mbd \ra \Delta_s$ is a $K$-quasiconformal mapping with $g(\bar{z}) = \overline{g(z)}$ for all $z \in \mbd .$
By the definition of $\varphi_1,$ it it easy to check that
$\diam (\varphi_1 (\mbs^1 \cap B(e^{i \pi},r))) \approx r^{K}$
whenever $r \ll 1.$
Together with \eqref{prop build qc3} and the bi-Lipschitz property of $\varphi_2$, this gives \eqref{prop build qc1}.
\end{proof}

\begin{remark}\label{after prop build qc}
Let $\hat{\varphi}_1 (z) = |z+1|^{K^{-1} -1} (z+1) $ for $K \in (1,\infty).$
Analogously to Theorem \ref{prop build qc}, replacing $\varphi_1$ by $\hat{\varphi}_1$, one can show that there exists a $K$-quasiconformal mapping $g\colon \mbd \ra \Delta_s $ satisfying that $g(\bar{z}) = \overline{g(z)}$ for all $z \in \mbd $
and $\diam (g (\mbs^1 \cap B(e^{i \pi},r))) \approx r^{2/K}$
whenever $r \ll 1.$
\end{remark}

Let $g$ be the quasiconformal mapping from Example \ref{prop build qc}.
In the following lemma, we give an upper bound for the Sobolev exponent of the inverse of extensions of $g$. 
\begin{lemma}\label{lem4.4}
Let $g$ be as in Example \ref{prop build qc}.
For any $f \in \mcf_s (g),$ if $f^{-1} \in W^{1,p} _{\loc} (\mbr^2 ,\mbr^2)$ for some $p \ge 1$, then necessarily $p<\frac{2(s+2)}{2s +2-K^{-1}} .$
\end{lemma}

\begin{proof}
Given a constant $c>0,$ we let $I_x = \{(x,y) : y \in [-|x|^{s+1} , |x|^{s+1}]\}$ for $x \in (0,c) .$
Via the ACL-property of Sobolev functions, we have that 
$$\mbox{osc}_{I_x} f^{-1} \le \int_{I_x} |D f^{-1}(x,y)|\, dy$$ 
for almost every $ x \in (0,c) .$
Jensen's inequality then implies that
\begin{equation}\label{negative part of f^-1: 2}
\frac{(\mbox{osc}_{I_x} f^{-1})^p}{x^{(s+1)(p-1)}} \le \int_{I_x} |D f^{-1} (x,y)|^p \, dy.
\end{equation}
Notice that by \eqref{prop build qc1} we have $\mbox{osc}_{I_x} f^{-1} \approx x^{1/(2K)} .$
Hence via Fubini's theorem we obtain from \eqref{negative part of f^-1: 2} that 
\begin{equation*}\label{infsupD_f^-1 10-1}
\int_{0} ^{c} x^{\frac{p}{2K} - (s+1)(p-1)} dx \lesssim \int_{B (0, 1)} |Df^{-1}|^p .
\end{equation*}
Therefore by the Sobolev assumption of $f^{-1},$ we have that $\frac{p}{2K} - (s+1)(p-1) >-1$, that is, $p<2(s+2)/(2s+2-K^{-1}).$
\end{proof}

\begin{remark}\label{rmk4.7}
Take $g \in \mcg_s(K)$ and $f \in \mcf_s (g).$
Notice that by \cite[Lemma 4.2]{gkt13}
$$\mbox{osc}_{I_x} f^{-1}  \gtrsim  x ^{K/(2-\epsilon)} $$
for all $x \in (0,c).$
Analogously to Lemma \ref{lem4.4}, if $f^{-1} \in W^{1,p} _{\loc} (\mbr^2 ,\mbr^2)$ for $p \ge 1$, 
then $p \le 2(s+2)/(2s+2 -K)_+ .$ 
\end{remark}

\begin{proof}[Proof of Theorem \ref{thm1}]
First of all, Lemma \ref{lem5.25} shows that $\mcf_s (g) \neq \emptyset$ for any $g \in \mcg_s (K).$
By \eqref{5.25-3}, it is obvious that $2K/(K-1)$ is an upper bound of the supreme in \eqref{supD_f}. Moreover by \eqref{lem5.25-1} we obtain that this supreme equals $2K/(K-1).$ This proves \eqref{supD_f}.
By \eqref{lem5.25-2} we see that $2(2+s)/(2s +2 -K^{-1})$ is a lower bound for the infimum in \eqref{infsupD_f^-1}.
Together with Lemma \ref{lem4.4} we may conclude \eqref{infsupD_f^-1}.

It remains to prove \eqref{supsupD_f^-1}.
By Remark \ref{rmk4.7} and \eqref{5.25-3}, 
we obtain that the minimum in \eqref{supsupD_f^-1} is an upper bound for the supreme in \eqref{supsupD_f^-1}.
Let $g$ be as in Remark \ref{after prop build qc}.
By the analogous proof of $\mathcal{F}_s (g)\neq \emptyset $ for Theorem \ref{theorem K_f},
we construct by hand a $f \in \mcf_s (g)$ satisfying that
$f^{-1} \in W^{1,p} _{\loc}(\mbr^2 ,\mbr^2)$ for all $p< \min \left\{2K/(K-1),2(s+2)/(2s+2 -K)_+ \right \} .$
We leave the details to interested readers. Hence the minimum in \eqref{supsupD_f^-1} is a lower bound for the supreme in \eqref{supsupD_f^-1}. The proof is therefore complete.
\end{proof}

Comparing Theorem \ref{thm1} with \cite[Theorem 1.2]{x19}, we miss the optimal regularity of distortions of extension and its inverse. We formulate the missing part of Theorem \ref{thm1} as a conjecture below.

\begin{conjecture}\label{conj1}
Let $\mcg_s (K)$ be as in \eqref{g_sK} and $\mcf_s (g)$ be as in \eqref{f_sg}. We conjecture that the following equations hold:
\begin{equation*}\label{infsupK_f}
\inf_{g \in \mcg_s (K)} \sup \{q \in (0,\infty) : f \in \mcf_s (g) ,\ K_f \in L^q _{\loc} (\mbr^2)  \} 
 =    \max \left\{\frac{1}{ K s}  , 1\right\} ,
\end{equation*} 
\begin{equation*}\label{supsupK_f}
\sup_{g \in \mcg_s (K)} \sup \{q \in (0,\infty) : f \in \mcf_s (g) ,\ K_f \in L^q _{\loc} (\mbr^2)  \} 
 =   \max \left\{\frac{K}{s}  , 1 \right\}, 
\end{equation*}
\begin{equation*}\label{infsupK_f^-1}
\sup \{q \in (0,\infty) : f \in \mcf_s (g) ,\ K_{f^{-1}} \in L^q _{\loc} (\mbr^2)  \} 
 =  \frac{2+s}{s} 
\end{equation*} 
for any $g \in \mcg_s (K).$
\end{conjecture}

Note that Conjecture \ref{conj1} is closely related to Theorem \ref{thm:inner cusp domain}.
 As $\Delta_s$ from \eqref{innercuap} is a special example of the more general class of John disks, it is nature to pose the following question.
\begin{question}
Study analogous results in Theorem \ref{thm1} when $\Delta_s$ is replaced by a John disk.  
\end{question}

\subsection*{Acknowledgements}
This survey is dedicated to our former Ph.D supervisor Prof.~Pekka Koskela for his excellent guidance during our postgraduate studies and for bringing us to the world of quasiconformal analysis.

\end{document}